\documentclass[11pt,reqno]{amsart}

\usepackage[margin=1in]{geometry}
\usepackage[all]{xy}
\usepackage{tikz}
\usepackage{pgfplots}
\usepackage{tikz-cd}
\usepackage{todonotes}
\usepackage{comment}
\usepackage[colorlinks=true]{hyperref}
\usepackage{enumitem}

\usepackage{extarrows}

\usepackage{stmaryrd}

\let\oldoverline\overline
\def\overline#1{\oldoverline{#1}\vphantom{#1}}

\newcommand{\vfc}[1]{{[}#1{]}^{vir}}

\usepackage{mathrsfs}
\usepackage{mathabx,graphicx}
\usepackage{amsmath}
\newcommand{\action}{\:\rotatebox[origin=c]{-90}{$\circlearrowright$}\:}

\theoremstyle{definition}
\newtheorem{theorem}{Theorem}[section]
\newtheorem{lemma}[theorem]{Lemma}

\newtheorem{corollary}[theorem]{Corollary}
\newtheorem{proposition}[theorem]{Proposition}

\newtheorem{definition}[theorem]{Definition}
\newtheorem{example}[theorem]{Example}
\newtheorem{remark}[theorem]{Remark}

\newcommand{\Div}{\text{Div}}

\def\Spec{\operatorname{Spec}}
\def\rank{\operatorname{rank}}
\def\Sym#1#2{[\operatorname{Sym}^{#1}#2]}

\newcommand{\Aut}{\underline{\text{Aut}}}

\newcommand{\Supp}{\text{Supp}}

\newcommand{\PP}{\mathbb{P}}

\newcommand{\ZZ}{\mathbb{Z}}

\newcommand{\CC}{\mathbb{C}}

\newcommand{\NN}{\mathbb{N}}

\newcommand{\OO}{\mathcal{O}}

\newcommand{\Mprel}{{\mathfrak{M}}}
\newcommand{\Ms}{\overline{M}}
\newcommand{\num}[1]{{\langle #1 \rangle}}

\newcommand{\Msprels}{{\overline{\mathfrak{M}}}}

\renewcommand{\frak}[1]{\mathfrak{#1}}

\newcommand{\pb}{\ar[dr, phantom, very near start, "\ulcorner"]}
\newcommand{\lpb}{{\arrow[dr, phantom, very near start, "\ulcorner \ell"]}}

\newcommand{\Tc}{\mathscr{T}}

\newcommand{\Gt}{\mathbb{G}_m^{trop}}

\newcommand{\MprelSd}{\widetilde{\mathfrak{M}}_{0,n}(BS_d)}
\newcommand{\MsSym}{\widetilde{M}_{0, n}(\Sym{d} X)}

\newcommand{\MprelSdG}{\widetilde{\mathfrak{M}}_\Xi(BS_d)}
\newcommand{\MsSymG}{\widetilde{M}_\Xi(\Sym{d} X)}

\renewcommand{\tilde}[1]{\widetilde{#1}}

\title{Costello's pushforward formula: errata and generalization}
\author{Leo Herr, Jonathan Wise}
\date{\today}

\begin{document}

\maketitle

\begin{abstract}
Costello's pushforward formula relates virtual fundamental classes of virtually birational algebraic stacks.  Its original formulation omits a necessary hypothesis, whose addition is not sufficient to correct the proof.  We supply a substitute for Costello's notion of pure degree and prove the pushforward formula with this definition.

We also show the hypotheses of the corrected pushforward formula are satisfied in a variety of its applications. Some adjustments to the original proofs are required in several cases, including the original one.
\end{abstract}

\setcounter{section}{0}
\section{Introduction}

If $f : X' \to X$ is a proper, birational morphism of varieties, then the fundamental class of $X'$ pushes forward to the fundamental class of $X$.  Birationality can be relaxed to generic finiteness of degree $d$, in which case $f_\ast [X'] = d [X]$.  Costello's pushforward formula asserts the same holds for virtual fundamental classes in a situation that might be called ``virtual birationality'':

\begin{theorem}[Costello's Pushforward Formula] \label{thm:costello}
Suppose there is a cartesian diagram
\begin{equation} \label{eqn:square}
\begin{tikzcd}
X' \ar[r, "f"] \ar[d, "p'", swap] \pb      &X \ar[d, "p"]      \\
Y' \ar[r, "g", swap]      &Y
\end{tikzcd}
\end{equation}
such that
\begin{enumerate}
\item $X'$ and $X$ are Deligne--Mumford stacks;
\item $Y'$ and $Y$ are Artin stacks of the same pure dimension;
\item $g$ is a morphism of Deligne--Mumford type and \emph{pure degree}~$d$;
\item $f$ is proper;
\item $p$ has a perfect relative obstruction theory $E$ inducing a perfect relative obstruction theory $E'$ for $p'$ by pullback.
\end{enumerate}
Then $f_\ast [X'/Y']^{\rm vir} = d [X/Y]^{\rm vir}$.
\end{theorem}

At issue is the definition of pure degree.  Costello defines a morphism $g : Y' \to Y$ of Deligne--Mumford type to be of pure degree $d$ if both $Y'$ and $Y$ have the same pure dimension and all of the generic fibers of $Y'$ over $Y$ are finite of degree $d$.  With this definition of pure degree, Theorem~\ref{thm:costello} is false:  see Examples~\ref{ex:costellocounterexbugeyedline} and~\ref{ex:costellocounterexgmvsa1}.  However, Manolache shows that the formula is true if either $g$ is projective or $g$ is proper and $Y$ is a Deligne--Mumford stack~\cite[Proposition~5.29 and Remark~5.30]{Manolache}.

In Section~\ref{sec:puredeg} we prove that Costello's original statement is valid, provided pure degree is defined as in Definition~\ref{def:purity}.  Our definition includes all diagrams~\eqref{eqn:square} in which $g$ is proper, which is easier to verify in practice than projectivity.


We found almost twenty papers that used Costello's pushforward formula, including several by the second author. The bulk of the present paper is devoted to checking that the relevant maps are indeed proper to ensure the formula was used correctly. This list is not meant to be exhaustive, but representative of techniques used to remedy the situation.

\subsection{Acknowledgments}

We would like to thank Kevin Costello for his correspondence and blessing to write this document. The first author would also like to thank Sarah Arpin, Gebhard Martin, Fabio Bernasconi, Tommaso de Fernex, Y.P. Lee, You-Cheng Chou, and the Math Overflow community \cite{leoherrmathoverflowpost}, \cite{380120}. Both authors would like to thank Dan Abramovich for useful correspondence and encouragement. 

This work grew out of a thesis \cite{logprodfmlaherr} written at the University of Colorado Boulder by the first author and supervised by the second author. The first author was partially funded by NSF RTG grant \#1840190. The second author was supported by a Collaboration Grant from the Simons Foundation.

\section{Pure degree and the pushforward formula}
\label{sec:puredeg}

The following two examples of squares \eqref{eqn:square} show some properness assumption is necessary for the pushforward formula to hold.

\begin{example}\label{ex:costellocounterexbugeyedline}
Let $Y$ be the affine line, $Y'$ the affine line with a doubled origin, and $Y' \to Y$ the projection that is the identity on each copy. Let $X$ be the origin of $Y$. The morphisms $p$ and $p'$ are local complete intersection embeddings, so their canonical obstruction theories induce virtual fundamental classes that are the ordinary fundamental classes. The map $g$ has pure degree~$1$ but $f$ has degree~$2$.
\end{example}

\begin{example}\label{ex:costellocounterexgmvsa1}
Let $Y$ be the affine line and $X$ its origin.  Let $Y'$ be the disjoint union of $\mathbb A^1$ and $\mathbb A^1 - \{ 0 \}$.  Then $Y' \to Y$ has pure degree~$2$ but $X' \to X$ is an isomorphism, hence has pure degree~$1$.
\end{example}

On the other hand, it would be too much to insist that $g$ actually be proper.  For example, $Y'$ might have a component that is not proper over $Y$ but is sufficiently far away from the image of $p$ so as not to affect $[X/Y]^{\rm vir}$. We propose the following definition of pure degree:

\begin{definition} \label{def:purity}
We say that a Deligne--Mumford type morphism of locally noetherian Artin stacks $g : Y' \to Y$ is \emph{pure} along $p : X \to Y$ if, whenever $S$ is the spectrum of a discrete valuation ring with closed point $s$, and $f : S \to Y$ is a morphism such that $f(s)$ lies in the image of $X$, the base change $Y'_S \to S$ is proper.
\end{definition}

We elaborate on the phrase ``of degree $d$ over'' in Appendix~\ref{sec:degofgenfin}. 

\begin{remark} \label{rem:raynaud}
Definition~\ref{def:purity} is related, but not equivalent, to the definition of Raynaud and Gruson~\cite[D\'efinition~(3.3)]{RG}. The map $Y' \to Y$ in Example~\ref{ex:costellocounterexgmvsa1} fails to be pure by either definition, while Example~\ref{ex:costellocounterexgmvsa1} is pure according to \cite[D\'efinition~(3.3)]{RG} but not according to Definition~\ref{def:purity}.
\end{remark}

\begin{proposition}
With notation as in the statement of Theorem~\ref{thm:costello}, the map of relative intrinsic normal cones $C_{X'/Y'} \to C_{X/Y}$ is of degree $d$ over each generic point of $C_{X/Y}$.
\end{proposition} 
\begin{proof} 
This assertion is local in $C_{X/Y}$, so it is also local in $X$ and $Y$. Replace both by smooth covers to assume $X, Y$ are affine schemes. 

The morphism $p : X \to Y$ can be factored as a closed embedding followed by a smooth morphism: $X \to \tilde Y \to Y$.  Then $C_{X/Y}$ is the stack quotient of $C_{X/\tilde Y}$ by $T_{\tilde Y/Y} \mathop\times_{\tilde Y} X$; likewise $C_{X'/Y'}$ is the quotient of $C_{X'/\tilde Y'}$ by $T_{\tilde Y/Y} \mathop\times_{\tilde Y} X'$ (where $\tilde Y'$ is the base change of $\tilde Y$ to $Y'$).  The generic fibers of $C_{X'/Y'}$ over $C_{X/Y}$ have the same degrees as the generic fibers of $C_{X'/\tilde Y'}$ over $C_{X/\tilde Y}$.  We may therefore replace $Y$ by $\tilde Y$ and assume that $p : X \to Y$ is a closed embedding.

Let $M' \to M$ be the morphism of deformations to the normal cone induced by the commutative diagram~\ref{eqn:square}.  Recall that $M$ is the complement of the strict transform of $Y \times \{ 0 \}$ in the blowup of $Y \times \mathbb A^1$ along $X \times \{ 0 \}$.  The normal cone $C = C_{X/Y}$ is the fiber of $M$ over $0 \in \mathbb A^1$.

Let $\xi$ be a generic point of $C$.  Let $R$ be the integral closure of $\mathcal O_{M,\xi}$.  Then $R$ is a $1$-dimensional, integrally closed, noetherian local ring, hence is a discrete valuation ring.  By construction, the composition of $\Spec R \to M \to Y$ sends the closed point of $\Spec R$ to $X \subset Y$ and its open point to the generic point of $Y$ as in Definition \ref{def:purity}.

By assumption, this implies $Y'_R \to \Spec R$ is proper.  Then $M' \mathop\times_M \Spec R \to \Spec R$ is also proper: the valuative criterion requires a unique lift for a commutative diagram
\begin{equation*} \xymatrix{
\Spec K' \ar[rr] \ar[d] & & M' \ar[d] \ar[r] & Y' \ar[d] \\
\Spec R' \ar[r] \ar@{-->}[urr] & \Spec R \ar[r] & M \ar[r] & Y
} \end{equation*}
after a finite extension of $R'$, but we get the lift $\Spec R' \to Y'$ by the properness of $Y'_R \to \Spec R$.  This induces a map $\Spec R' \to Y' \times \mathbb A^1$ that factors through the blowup $\overline M'$ of $Y' \times \mathbb A^1$ along $X' \times \{ 0 \}$.  This map lies in $M' \subseteq \overline M'$ because the generic point of $\Spec R'$ maps a point of $Y' \times \mathbb A^1$ over the generic point of $\mathbb A^1$.  In particular, the closed point cannot lie in the strict transform of $Y \times \{ 0 \}$, so the image of $\Spec R' \to \overline M'$ is contained in $M$.

Let us write $M'_R = \Spec R \mathop\times_M M'$ (note that the fiber product is over $M$, not over $Y$).  We have just seen that $M'_R \to \Spec R$ is proper.  We argue that it is also flat.  It suffices to show $M'_R$ is torsion free.  But under the map $\Spec R \to M \to \mathbb A^1$, a uniformizer $t$ of $\mathbb A^1$ at the origin pulls back to a nonzero element of $R$, which is a power of the maximal ideal of $R$, since $R$ is a discrete valuation ring.  By construction of $M'$, it has no $t$-torsion, so $M'_R$ must be torsion-free over $R$.

Now $M'_R$ is a proper and flat Deligne--Mumford stack over $\Spec R$.  It remains only to show that the fibers of $M'_R$ have the same degree over $\Spec R$.  We can replace $R$ with a flat cover, so we assume $R$ is complete.  Let $U \to M'_R$ be an \'etale cover with $U$ affine.  Then $U$ is~$1$-dimensional, flat, of finite type over $\Spec R$.  Therefore it is quasifinite over $R$.  Since $R$ is complete, $U = U_0 \sqcup V$ where $U_0$ is finite over $R$ and the closed fiber of $V$ is empty.  We can replace $U$ by $U_0$ and then $U \to M'_R$ and $U \to \Spec R$ are both finite and flat.  Note $U \to M'_R$ is finite because $M'_R$ is proper over $\Spec R$.

Assume without loss of generality that $M'_R$ is connected.  If $d$ is the degree of $M'_R$ over the generic fiber of $\Spec R$ then $d = \rank_R \mathcal O_U / \rank_{\mathcal O_{M'_R}} \mathcal O_U$, which is the same whether evaluated at the generic or the special point of $\Spec R$.  On the special fiber, this ratio is the multiplicity of the pushforward of $C_{X'/Y'}$ at the point $\xi$. On the generic fiber, it is the pure degree of $Y'$ over $Y$, as required.
\end{proof}

\begin{proof}[Proof of Theorem~\ref{thm:costello}]
Let $E^\vee$ and ${E'}^\vee$ denote the vector bundle stacks dual to the obstruction theories $E$ and $E'$.  As relative obstruction theories, there are closed embeddings $C_{X/Y} \subset E^\vee$ and $C_{X'/Y'} \subset {E'}^\vee$. Their compatibility entails a commutative diagram whose lower square is cartesian:
\begin{equation*}
\begin{tikzcd}
C_{X'/Y'} \ar[r] \ar[d] & C_{X/Y} \ar[d] \\
{E'}^\vee \ar[r, "h"] \ar[d,"q'", swap] \pb & E^\vee \ar[d, "q"] \\
X' \ar[r, "f"] & X
\end{tikzcd}
\end{equation*}
By compatibility of proper pushforward and flat pullback, we have
\begin{equation}
q^\ast f_\ast [X'/Y']^{\rm vir} = h_\ast {q'}^\ast [X'/Y']^{\rm vir} = h_\ast [C_{X'/Y'}] = d [C_{X/Y}] = q^\ast\bigl( d [X/Y]^{\rm vir} \bigr)
\end{equation}
But $[X/Y]^{\rm vir}$ is the unique cycle class on $X$ such that $q^\ast [X/Y]^{\rm vir} = [C_{X/Y}]$, so we conclude that $f_\ast [X'/Y']^{\rm vir} = [X/Y]^{\rm vir}$, as required.  
\end{proof}

\begin{remark}\label{rmk:puredegprops}
We record some consequences of Definition \ref{def:purity}. 
\begin{itemize}
    \item If a map $Y' \to Y$ is proper, it is pure along any morphism $X \to Y$. 
    
    \item If $\widetilde{X} \to X$ is surjective and $Y' \to Y$ is pure along $\widetilde{X} \to X \to Y$, then it's also pure along $X \to Y$. 
    
    \item If $Y' \to Y$ is of pure degree along a map $X \to Y$, then it is of pure degree along any map $Z \to X \to Y$. 
    
    \item Suppose in Diagram \eqref{eqn:square} that $p, p'$ are open immersions. If $f$ is proper, then $g: Y' \to Y$ is of pure degree along $p$ for topological reasons. 
    
    \item Purity is stable under base change in $Y$. 
    
	\item Purity of $g$ in diagram~\eqref{eqn:square} implies properness of $f$, so the assumption that $f$ be proper is redundant.
\end{itemize}
\end{remark}

\begin{remark}
In this paper, we verify the hypotheses of Theorem~\ref{thm:costello} by showing $g$ is proper.  We did find applications of Costello's theorem in the literature where $g$ was pure but not proper, but we found it easier to replace $Y'$ by a smaller stack that was proper than to verify purity directly.
\end{remark}

\section{Higher genus stable maps and genus zero orbifold stable maps}
\label{sec:highergenus}

Let $X$ be a smooth, projective scheme and work over $\Spec \CC$. The original application of the pushforward formula was to the following fiber square in the proof of~\cite[Lemma~8.0.2]{Costello}: 

\begin{equation}\label{eqn:dfoldcovercostellosquare}
\begin{tikzcd}
\overline{M}_\eta(X) \arrow[r] \arrow[d] \arrow[dr, phantom, very near start, "\ulcorner"]        &\overline{M}_v(X) \arrow[d]      \\
\mathfrak{M}_\eta \arrow[r]       &\mathfrak{M}_v.
\end{tikzcd}
\end{equation}
The stack $\mathfrak{M}_\eta$ parametrizes finite, \'etale, $d$-sheeted covers $C' \rightarrow C$ with fixed numerical data, and the horizontal arrow sends such a cover to its source curve $C'$. The stack $\overline{M}_\eta(X)$ is defined to make this square cartesian~\cite[pp.\ 575, 591, 593]{Costello}.\footnote{In the fiber products on pp.\ 575 and 591, $\overline{\mathcal M}_{s(\eta)}$ and $\overline{\mathcal M}_{r(s(\eta))}$ were presumably meant to be $\overline{\mathcal M}_{s(\eta)}(X)$ and $\overline{\mathcal M}_{r(s(\eta))}(X)$, respectively.}

The next example illustrates that the horizontal arrows in Diagram~\ref{eqn:dfoldcovercostellosquare} are not proper, and therefore that $\overline{M}_\eta(X)$ is not proper.  Since stable maps to $\Sym d X$ do form a proper Deligne--Mumford stack, $\overline M_\eta(X)$ cannot be one of its components, as claimed in \cite[Lemma~2.4.2]{Costello}.  The pushforward formula cannot be applied to $\overline M_\eta(X)$ because even its statement requires proper pushforward along the upper horizontal arrow $\overline M_\eta(X) \to \overline M_v(X)$.

\begin{example}
Let $X = \mathbb{P}^2$ and $U = \mathbb{A}^1_{\lambda} \setminus \{ 0 \}$. Consider the family of plane cubics $C'_\lambda \subseteq X$ indexed by $\lambda \in U$ given by the projective closure of
\[y^2 = x(x + \lambda)(x+1).\]

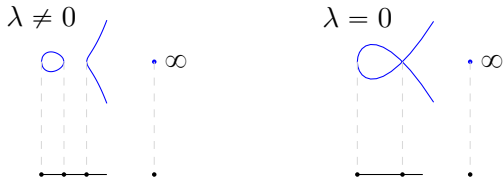
\begin{figure}
    \centering
    \begin{tikzpicture}[scale=.6]
	\begin{scope}[shift={(7,0)}]
	\node at (-1, 1){$\lambda = 0$};
	\draw[variable = \x, domain=-1.3:1.3, blue] plot (\x*\x-1, \x*\x*\x-\x); 
	\fill[blue] (1.5, 0) circle (.05);
	\node[right] at (1.5, 0){{\small $\infty$}};
	\draw (-1, -2.5)--(.45, -2.5);
	\foreach \t in {-1, 0, 1.5} {
	\draw[dashed, color=gray!30] (\t, 0) -- (\t, -2.5);
	\fill (\t, -2.5) circle (.05);}
	\end{scope}
	\node at (-1, 1){$\lambda \neq 0$};
	\draw[variable = \x, domain=0:.45, blue] plot (\x, sqrt{\x*(\x+.5)*(\x+1)}); 
	\draw[variable = \x, domain=-1:-.5, blue] plot (\x, sqrt{\x*(\x+.5)*(\x+1)}); 
	\draw[variable = \x, domain=0:.45, blue] plot (\x, -sqrt{\x*(\x+.5)*(\x+1)}); 
	\draw[variable = \x, domain=-1:-.5, blue] plot (\x, -sqrt{\x*(\x+.5)*(\x+1)}); 
	\fill[blue] (1.5, 0) circle (.05);
	\node[right] at (1.5, 0){{\small $\infty$}};
	\draw (-1, -2.5)--(.45, -2.5);
	\foreach \t in {-1, -.5, 0, 1.5} {
	\draw[dashed, color=gray!30] (\t, 0) -- (\t, -2.5);
	\fill (\t, -2.5) circle (.05);}
\end{tikzpicture}
    \caption{A family of maps to $\PP^1$ with general fiber a ramified cover and special fiber that is not. }
    \label{fig:degeneratingcubic}
\end{figure}

We will describe a modification of $\frak M_\eta$ that makes the horizontal arrows proper and revives \cite[Lemma~2.4.2]{Costello}. Assume the graph $\eta$ has a single vertex and eliminate the graphs from the notation for simplicity.  We leave it to the reader to adapt the method to more complicated graphs and deduce Costello's main theorem in its original form.

The closure of the projection $[x:y:z] \mapsto [x:z]$ gives a map of curves $C'_\lambda \to \mathbb{P}^1 \times U$ over $U$. Mark the four sections of $\mathbb{P}^1 \times U$ given by the branch locus and endow them with $B \mathbb{Z}/2$-stack structure. This yields a family of stacky projective lines over $U$ which we call $C_\lambda$. The map $C'_\lambda \to C_\lambda$ is a proper, \'etale, 2-sheeted cover of stacky curves. Degenerating the source to $\lambda = 0$, we see there's no way to add stack structure to the base $\mathbb{P}^1$ to complete this family to an \'etale $\mathbb{Z}/2$-torsor. 

The map $C'_\lambda \to C_\lambda$ is classified by a map $C_\lambda \to BS_2$, which is classified in turn by a map $U \to \Ms_{0, 4}(BS_2)$. We will instead take the limit in the sense of twisted stable maps (equivalently, admissible covers), which allows $C_\lambda$ to degenerate into two copies of $\mathbb P^1$ joined at a node.

\end{example}

Write $S_d$ for the symmetric group on $d$ letters and $\num{d} := \{1, 2, \dots, d\}$. Our solution is to use twisted stable maps to the stack $\Sym{d}X = [X^d/S_d]$. A map $C \to \Sym{d} X$ may be interpreted equivalently as a torsor $S_d \action P \to C$ with an equivariant map $P \to X^d$ or a $d$-sheeted finite \'etale cover $C' \to C$ given by $C' = \num{d} \times^{S_d} P$ with a map $C' \to X$ \cite[Lemma 2.2.1]{Costello}. We will apply Costello's pushforward formula to the cartesian diagram of moduli stacks of twisted curves after comparing conventions for torsors and twisted stable maps.

The map $S_{d-1} \to S_d$ including those permutations which fix the last element induces a group action $S_{d-1} \action S_d$. Consider the $d$-element set $\num{d} \simeq S_{d-1} \backslash S_d$ as a set-theoretic quotient with \emph{right}-action by $S_d$. We may view $BS_{d-1} \simeq \num{d}/S_d$ similarly. This latter identification does not depend on which $(d-1)$-element subset $S_{d-1}$ is allowed to act.

\begin{lemma}
If $T \to BS_d$ classifies an $S_d$-torsor $P \to T$, then the contracted product and fiber product are the same: 
\[\num{d} \operatorname\times^{S_d} P \simeq T \times_{BS_d} BS_{d-1}.\]

\end{lemma}

Applying this lemma to the torsor $X^d \to \Sym{d} X$, the associated $d$-sheeted cover is
\[X \times \Sym{d-1} X \simeq [ X^d/S_{d-1} ] \simeq \Sym{d} X \times_{BS_d} BS_{d-1}.\]
This comes with a canonical projection to $X$.

\begin{remark}
Our description of the equivalence between the categories of $S_d$-torsors and $d$-sheeted covers over $T$ is the opposite of \cite[Lemma 2.2.1]{Costello}.
\end{remark}

\begin{remark}\label{rmk:trivialcostellogerbes}
A stack $\mathcal K_{0, n}(V)$ of twisted stable maps to a smooth projective target $V$ was defined in \cite{abramcortivistoli} that allows marked points to have nontrivial gerbe structure. We instead use the stack $\Ms^{tw}_{0, n}(V)$ that requires those gerbes to be trivialized by sections at each marked point, as in \cite[\S 2]{Costello}. Our marked points are \emph{globally} of the form $B \mu_r$ for some ``ramification order'' $r \in \ZZ_{\geq 1}$. We demand similarly that the gerbes of relative twisted stable maps $\Ms_{0, n}^{tw}(V/W)$ \cite[\S 8.3]{twistedstablemaps} for a map $V \to W$ be trivialized. 

The map $\Ms^{tw}_{0, n}(V) \to \mathcal K_{0, n} (V)$ is the universal gerbe, of degree $\dfrac{1}{r_1 \cdot r_2 \cdot \cdots r_n}$ over the locus where the gerbes are $B\mu_{r_i}$. 
\end{remark}

\begin{remark}\label{rmk:twistedstabmapsfixramification}

A $d$-fold cover of orbifold curves $C' \to C$ over $S$ has discrete invariants including the genera, the stack structures at marked points, and the maps $B \mu_{r_i} \to BS_d$ from each $i$th marked point of $C$ encoding its fiber in $C' \to C$. All are locally constant in $S$. If the fiber over the $i$th marked point is denoted $\num{\ell_i}$, there is a function $\tau : \num{\ell_i} \to \ZZ_{\geq 1}$ sending each point to its ramification order.  The function depends in a locally constant fashion on the cover $C' \to C$.

Let $\Xi$ be a monodromy profile, specified by $n$ maps $B \mu_{r_i} \to BS_d$ parametrizing the fiber over each marked point as a $d$-sheeted cover of $B\mu_{r_i}$ (specifying, in other words, the monodromy of the cover around the $i$th parked point of the base). The substack $\Ms_\Xi^{tw}(\Sym d X) \subseteq \Ms_{0, n}^{tw}(\Sym d X)$ of stable maps from covers $C' \to C$ with monodromy profile $\Xi$ is open and closed. The fiber over the $i$th marked point consists of $\ell_i$ ramified points. Choosing an ordering of each fiber among $S_{\ell_i}$ choices makes the source $C'$ into a $\ell = \sum \ell_i$-marked curve. There is an open and closed substack inside $\Ms^{tw}_\Xi(\Sym d X)$ that also fixes $\tau$. If $\tau$ is increasing for example, $r_i > r_j$ implies $i > j$ and ramified points come later in the ordering.

The monodromy profile $\Xi$ is a component of the cyclotomic intertia stack $I_\mu(BS_d)$ \cite[Definition 3.2.1]{GWthyofstacksAbramGraberVistoli}, recipient of evaluation maps 
\[\Ms_{0, n}^{tw}(\Sym d X) \to (I_\mu(\Sym d X))^n \to (I_\mu(BS_d))^n.\]

\end{remark}

\subsection{Applying the pushforward formula to the new diagram}

Fix nonnegative integers $g$, $n$, $d$, $\ell = \sum \ell_i$ with $n \leq \ell \leq dn$ and monodromy profile $\Xi$ as in Remark \ref{rmk:twistedstabmapsfixramification}. We are ready to reinterpret \eqref{eqn:dfoldcovercostellosquare}:
\begin{equation}\label{eqn:pbmodcurvestacks}
\begin{tikzcd}
\tilde M_\Xi(\Sym d X) \ar[r, "q"] \ar[d, "\pi'"] \ar[dr, phantom, very near start, "\ulcorner"]     &\Ms_{g, \ell}(X) \ar[d, "\pi"]     \\
\tilde {\frak M}_{\Xi}(BS_d) \ar[r, "p"]      &\Mprel_{g, \ell}.
\end{tikzcd}\end{equation}

The stacks $\Ms_{g, \ell}(X), \Mprel_{g, \ell}$ parametrize ordinary stable maps to $X$ and prestable curves with $\ell$ marked points, while the map $\pi$ forgets all but the source curve of the stable map. We now introduce $p$ in three steps. 

\textit{Step 1: Relative Maps}

Let $u : \mathfrak D \to \Mprel_{g, \ell}$ be the universal curve and $\Ms^{tw}_{0, n}(u) = \Ms^{tw}_{0, n}(\Sym d {\mathfrak D} / \Mprel_{g, \ell})$ be the stack of relative twisted stable maps with trivialized marked gerbes, as described  in Remark \ref{rmk:trivialcostellogerbes}. If $S \to \Mprel_{g, \ell}$ classifies a curve $D \to S$, points of this stack are given by:
\[\begin{tikzcd}
        &\Ms^{tw}_{0, n}(u) \ar[d]         \\
S \ar[r, "D", swap] \ar[ur, dashed]       &\Mprel_{g, \ell}
\end{tikzcd}    :=
\left\{\begin{tikzcd}[column sep=small]
C \ar[rr, "\widehat{f}"] \ar[dr]      &    &\Sym d D \ar[dl]      \\
    &S        &
\end{tikzcd}\middle|\quad \parbox{2.5in}{$C$ is a connected, nodal orbifold curve \\ with trivialized marked gerbes, and \\ $\widehat{f}$ is representable and stable} \right\}.\]
Maps $C \to \Sym d D$ are equivalent to $d$-sheeted finite \'etale covers $C' \to C$ of twisted curves with a representable map $C' \to D \times \Sym {d-1} D$. Stability requires that the sheaf of automorphisms of the trio $(C \leftarrow C' \to D)$ that restrict to the identity on $D$ be finite at geometric points. 

\textit{Step 2: Marked Points}

Endow $C'$ with the marked points pulled back from those of $C$. These preimages are \emph{unordered}, so $C'$ is not yet a marked curve. 

If $C' \to C$ were an \emph{untwisted} finite \'etale cover, order the $d$ preimages of each marked point of $C$. Then $\ell_i = d$, $\ell = dn$, and ordering amounts to a $(S_d)^n$-torsor on moduli spaces in the untwisted case. Globally fixing a lexicographic ordering of $\num{n} \times \num{d}$ then equates a $\num{n} \times \num{d}$-marked curve with a $dn$-marked curve.

For twisted/ramified covers $C' \to C$, the sizes $\ell_i$ of the fibers over the marked points vary according to $\Xi$. Ordering them nevertheless entails a $\prod S_{\ell_i}$-torsor $\widetilde{M}_{0, n}(u)$ over $\Ms_{0, n}^{tw}(u)$, as in Remark \ref{rmk:twistedstabmapsfixramification}.

Now $\widetilde{M}_{0, n}(u)$ parametrizes trios of curves $(C \leftarrow C' \to D)$ where $C'$ is a marked curve, but $C' \to D$ needn't be a map of marked curves. This condition cuts out a closed substack $\widetilde{M}'_{0, n}(u) \subseteq \widetilde{M}_{0, n}(u)$ where $C' \to D$ maps marked points to marked points in order.

\textit{Step 3: Partial Stabilization}

Think of the map $f : C' \to D$ as a prestable twisted map to $D$ and take its stabilization $\overline{f} : \overline{C}' \to D$. If $\overline{f}$ identifies $D$ with the coarse moduli space of $\overline{C}'$ then $f$ is called a \textit{partial stabilization}. Define 
\[\MprelSd \subseteq \widetilde{M}'_{0, n}(u)\]
to be the substack where $f$ is a partial stabilization. This substack is open and closed:

\begin{lemma} \label{lem:open-and-closed} \label{lem:partstabnclopen}
Let $f : C' \to D$ be a morphism of untwisted prestable curves over a base $S$.  The locus in $S$ where $f$ is a partial stabilization is both open and closed.
\end{lemma}
\begin{proof}

The stabilization $\overline{f} : \overline{C}' \to D$ is an isomorphism over an open locus \cite[\href{https://stacks.math.columbia.edu/tag/05XD}{Tag 05XD}]{stacks-project}. This locus is also stable under specialization by the uniqueness of stable limits of stable maps.
\end{proof}

The map $p$ in Diagram~\eqref{eqn:pbmodcurvestacks} is the composite of all the proper maps above:
\[p : \MprelSd \subseteq \widetilde{M}'_{0, n}(u) \subseteq \widetilde{M}_{0, n}(u) \to \Ms^{tw}_{0, n}(u) \to \Mprel_{g, \ell}.\]

\begin{corollary}\label{cor:costellopisproper}
The map $p : \MprelSd \to \Mprel_{g, \ell}$ is proper, thus pure.
\end{corollary}
\begin{proof}
It was constructed as a closed substack of a finite cover of the space of stable maps to a target that is proper over $\Mprel_{g,\ell}$.  Since stable maps to a proper target form a proper space, this implies that $p$ is proper.
\end{proof}

An example due to Costello \cite[Lemma 6.0.1]{Costello} of $(g, n, d, \Xi)$ where $p$ is generically finite is worked out in \S \ref{subsec:costelloexampledegree}. 

The stack of twisted stable maps $\Ms_{0, n}(\Sym d X)$ parametrizes $(C \leftarrow C' \to X)$ as in Step 1 above. There is a similar $\prod S_{\ell_i}$-torsor $\MsSym$ over $\Ms_{0, n}(\Sym d X)$ of orderings of the preimages in $C'$ of the marked points of $C$. This is the remaining piece of Diagram \eqref{eqn:pbmodcurvestacks}. The map $q$ sends $(C \leftarrow C' \to X)$ to the stabilization $\overline{C}' \to X$, and $\pi'$ sends it to $(C \leftarrow C' \to \overline{C}')$.

\begin{remark}\label{rmk:stabnfunctorial}
Stabilization $s : C \to \overline{C}$ of prestable maps $C \to X$ is functorial in that $X$-automorphisms $\varphi : C \simeq C$ all lie over unique $X$-automorphisms $\overline{\varphi} : \overline{C} \simeq \overline{C}$. Automorphisms of $C \to X$ lying over an automorphism of $\overline{C}$ form a subsheaf:
\[i : \Aut_X(C \to \overline{C}) \subseteq \Aut_X(C) \quad \quad \quad 
\begin{tikzcd}
C \ar[d] \ar[r, "\sim"]       &C \ar[d]      \\
\overline{C} \ar[r, "\sim"]        &\overline{C}
\end{tikzcd} \mapsto (C \simeq C).
\]
To argue $i$ is an isomorphism, assume the base is a geometric point \cite[03PU]{stacks-project}. But $\varphi$ must restrict to an automorphism on the union of unstable components over $X$, so $C \overset{\varphi}{\simeq} C \to \overline{C}$ is also the stabilization of $C \to X$. The map $\overline{\varphi}$ comes from canonicity of stabilization. 

\end{remark}

\begin{lemma}\label{lem:pbmodcurvestacksiscartesian}
The square \eqref{eqn:pbmodcurvestacks} is cartesian. 
\end{lemma}

\begin{proof}
If $\MsSym$ and $\MprelSd$ and $\overline M_{g,\ell}(X)$ are all replaced by their unstable variants, Diagram~\eqref{eqn:pbmodcurvestacks} is immediately commutative and cartesian. Given 
\[C \leftarrow C' \to \overline{C}' \to X\]
with $\overline{C}' \to X$ stable but $(C \leftarrow C' \to \overline{C}')$ not necessarily, we must show $(C \leftarrow C' \to X)$ is stable if and only if $(C \leftarrow C' \to \overline{C}')$ is. This is a consequence of Remark \ref{rmk:stabnfunctorial}.

\end{proof}

Diagram \ref{eqn:pbmodcurvestacks} is cartesian and $p$ is proper. It remains only to compare the perfect obstruction theories of $\pi$ and $\pi'$.

\begin{lemma}
The map 
\[\MprelSd \to \Mprel_{0, n}(BS_d) \times_{(BS_d)^n} \ast\] 
forgetting the partial stabilization is \'etale.

\end{lemma}

\begin{proof}

Apply the formal criterion as in \cite[Lemma 7]{gwinvtsag} or {\cite[Lemma B (ii) for $\Upsilon$]{compthmsforgwisofsmpairs}}.

\end{proof}

\begin{corollary}
The perfect relative obstruction theories on $\pi'$ induced from the natural ones on $\pi$ and from $\Ms_{0, n}(\Sym d X) \to \Mprel_{0, n}(BS_d)$ coincide. 
\end{corollary}


\subsection{Costello's example computation}\label{subsec:costelloexampledegree}

We now compute the pure degree $e$ of $p : \tilde{\frak M}_{0, n}(BS_d) \to \Mprel_{g, \ell}$ in the setting of \cite[Lemma 6.0.1]{Costello} to check consistency.

Remark \ref{rmk:twistedstabmapsfixramification} lets one fix discrete data $(g, d, n, \ell_i, \Xi)$ illustrated in Figure \ref{fig:excostellodiscretedata}. Let $d = g+1$ and $k$ be an integer to be specified later. Require $n \leq \ell \leq dn$, as $\ell = \sum \ell_i$ counts marked points of the source $C'$ and $n$ of the target $C$.

Consider decompositions $\num{n} = \num{k} \sqcup J \sqcup \{\infty\}$ and $\num{\ell} = \num{k\cdot d}\sqcup J^d \sqcup I$ and a function $d : I \to \mathbb{Z}_{\geq 1}$, with $m(\infty):= \text{lcm}(d(i))$. Define 
\[\Xi = 
\left\{\begin{tikzcd}[row sep=small]
g(C') = g, \,g(C) = 0,        \\
\num{\ell} \to \num{n} \text{ is } \num{kd} = \num{k} \times \num{d} \overset{pr_1}{\mapsto} \num{k},\,\, J^d = J \times \num{d} \overset{pr_1}{\mapsto} J,\,\, I \mapsto \infty, \\
\forall j \in J,\,\, r_j = 1, \forall i \in \num{k}\,\, r_i = 2, r_{\infty} = m(\infty)       \\
\bigsqcup_J B\mu_1 = \ast \to BS_d,\,\, \bigsqcup_{i \in \num{k}} B\mu_2 \overset{\psi}{\to} BS_d,\,\, B \mu_{m(\infty)} \overset{\phi}{\to} BS_d
\end{tikzcd}\right\}\]
Here $\psi$ corresponds to the action $\mu_2 \action \num{d}$ with one 2-cycle and the other points fixed, while $\phi$ corresponds to the action $\mu_{m(\infty)} \action \num{d}$ with each orbit of size $\dfrac{m(\infty)}{d(i)}$ having stabilizer of size $d(i)$ (well-defined up to choice of ordering of $\num{d}$ in $\MprelSd$). These discrete data define a substack $\tilde M_\Xi(\Sym d X) \subseteq \MsSym$.

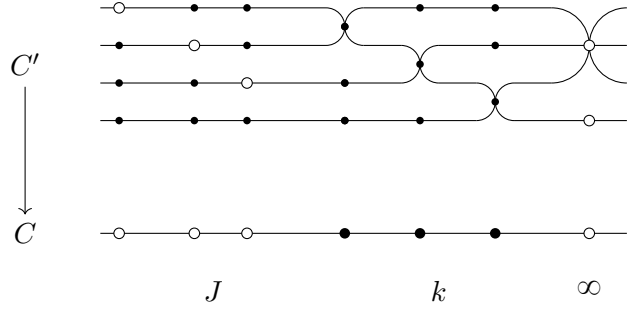
\begin{figure}
    \centering
\begin{tikzpicture}
\node at (-1, .75){$C'$};
\draw[->] (-1, .45) to (-1, -1.25);
\node at (-1, -1.5){$C$};
\draw[-] (0, 0) to (3, 0);
\draw[-] (0, 0.5) to (3, 0.5);
\draw[-] (0, 1) to (3, 1);
\draw[-] (0, 1.5) to (3, 1.5);
\draw (3, 1) arc (-90:90:.25);
\draw (3.5, 1.5) arc (90:270:.25);
\draw[-] (3.5, 1.5) to (6, 1.5);
\draw (4, .5) arc (-90:90:.25);
\draw (4.5, 1) arc (90:270:.25);
\draw (5, 0) arc (-90:90:.25);
\draw (5.5, 0.5) arc (90:270:.25);
\draw[-] (3.5, 1) to (4, 1);
\draw[-] (4.5, 1) to (7, 1);
\draw[-] (3, .5) to (4, .5);
\draw[-] (4.5, .5) to (5, .5);
\draw[-] (5.5, .5) to (6, .5);
\draw[-] (3, 0) to (5, 0);
\draw[-] (5.5, 0) to (7, 0);
\draw (6, 0.5) arc (-90:90:.5);
\draw (7, 1.5) arc (90:270:.5);
\draw[-] (0, -1.5) to (7, -1.5);
\node[below] at (1.5, -2){$J$};
\fill[white] (0.25, -1.5) circle (.07cm);
\fill[white] (1.25, -1.5) circle (.07cm);
\fill[white] (1.95, -1.5) circle (.07cm);
\draw (0.25, -1.5) circle (.07cm);
\draw (1.25, -1.5) circle (.07cm);
\draw (1.95, -1.5) circle (.07cm);
\fill (0.25, 0) circle (.05cm);
\fill (1.25, 0) circle (.05cm);
\fill (1.95, 0) circle (.05cm);
\fill (0.25, 0.5) circle (.05cm);
\fill (1.25, 0.5) circle (.05cm);
\fill[white] (1.95, 0.5) circle (.07cm);
\draw (1.95, 0.5) circle (.07cm);
\fill (0.25, 1) circle (.05cm);
\fill[white] (1.25, 1) circle (.07cm);
\draw (1.25, 1) circle (.07cm);
\fill (1.95, 1) circle (.05cm);
\fill[white] (0.25, 1.5) circle (.07cm);
\draw (0.25, 1.5) circle (.07cm);
\fill (1.25, 1.5) circle (.05cm);
\fill (1.95, 1.5) circle (.05cm);
\fill (3.25, 0) circle (.05cm);
\fill (3.25, 0.5) circle (.05cm);
\fill (3.25, 1.25) circle (.05cm);
\fill (3.25, 0) circle (.05cm);
\fill (3.25, 0.5) circle (.05cm);
\fill (3.25, 1.25) circle (.05cm);
\fill (4.25, 0) circle (.05cm);
\fill (4.25, 0.75) circle (.05cm);
\fill (4.25, 1.5) circle (.05cm);
\fill (5.25, 0.25) circle (.05cm);
\fill (5.25, 1) circle (.05cm);
\fill (5.25, 1.5) circle (.05cm);
\fill (3.25, -1.5) circle (.07cm);
\fill (4.25, -1.5) circle (.07cm);
\fill (5.25, -1.5) circle (.07cm);
\node[below] at (4.5, -2){$k$};
\fill[white] (6.5, 1) circle (.07cm);
\draw (6.5, 1) circle (.07cm);
\fill[white] (6.5, 0) circle (.07cm);
\draw (6.5, 0) circle (.07cm);
\fill[white] (6.5, -1.5) circle (.07cm);
\draw (6.5, -1.5) circle (.07cm);
\node[below] at (6.5, -2){$\infty$};
\end{tikzpicture}
    \caption{A cover in $\Xi$, $g=3, d=4$. Marked points are colored black if forgotten and white if remembered under the map to $\Mprel_{g, \ell-A}$}
    \label{fig:excostellodiscretedata}

\end{figure}

Choose $A \subseteq \num{k \cdot d} \sqcup J^d \subseteq \num{\ell}$ such that $\num{k \cdot d} \subseteq A$ and $J^d \setminus (A \cap J^d) \to J$ is a bijection. Write $\ell-A$ abusively for the set $\num{\ell} \setminus A = I \sqcup J$. Define a map $\pi : \tilde M_\Xi(\Sym d X) \to \Ms_{g, \ell-A}(X)$ as the composite of $p$ from Diagram \ref{eqn:pbmodcurvestacks} and the map forgetting the marked points $A$ and stabilizing. Define a stack $\MprelSdG$ as above to keep track of triples $(C \leftarrow C' \to D)$ with $C' \to D$ a partial stabilization \textit{after forgetting} the points labelled by $A$, $C' \to C$ a $d$-sheeted cover, and an ordering on the preimages of $C$'s marked points. This fits in a cartesian square
\[\begin{tikzcd}
\MsSymG \ar[r] \ar[d] \pb        &\Ms_{g, \ell-A}(X) \ar[d]       \\
\MprelSdG \ar[r, "p'"]       &\Mprel_{g, \ell-A}.
\end{tikzcd}\]    
One shows this square is cartesian as in Lemma \ref{lem:pbmodcurvestacksiscartesian} and that $p'$ is proper as in Corollary \ref{cor:costellopisproper}. We compute its pure degree after \cite[Lemma 6.0.1]{Costello}:

\begin{theorem}\label{thm:costellopushforwarddegreecalculation}
With the above discrete data $\Xi$ and $k = \# I + 3g -1$, $\dim \MprelSdG = \dim \Mprel_{g, \ell-A}$ and the map $p'$ between them is of pure degree
\[e = \dfrac{k! (g!)^{\# J} (g!)^k}{2^k \cdot m(\infty)}.\]
\end{theorem}

\begin{proof}

The map from $\MprelSdG$ that forgets $D$ is \'etale, so we can ignore $D$ to calculate its dimension. Our moduli spaces are the closure of strata considered by Costello, so we again have
\[\dim \MprelSdG = k + \#J - 2, \quad \quad \quad \dim \Mprel_{g, \ell-A} = 3g-3 + \#I + \#J.\]
These are equal by definition of $k$. Because the dimensions are equal, the preimage of the generic point must either be the generic point of the source or empty. The generic point of the source has smooth $C', C$ by design, hence $C' \overset{\sim}{\to} D$ is an isomorphism. We've reduced to the case considered by Costello.

Fix general points $q_1, \dots, q_s \in D$ and write $B = \sum d(i) [q_i]$ for the induced divisor of degree $g+1$.

\textbf{Claim:}
There are no special effective subdivisors $0 \leq B' < B$ of degree $g$.  

We outsource the proof to Lemma \ref{lem:nospecialeffectivesubdivs}. We conclude as in Costello's original argument. Any effective $B' < B$ of degree $g$ is not special, so $h^1(\OO(B')) = h^1(\OO(B)) = 0$. Riemann-Roch gives $h^0(\OO(B')) = 1$ and $h^0(\OO(B)) = 2$. This means there is at most one map $f : D \to \PP^1$ with $f^\ast \infty \leq B$ up to isomorphism and no such maps with $f^\ast \infty \leq B'$; i.e., $f^\ast \infty = B$.

The dimension of the moduli space of covers of $\PP^1$ is determined by the number of marked and branch points \cite[\S 1.G]{harris1998moduli}. A more-ramified cover has fewer branch points, so general maps $f : D \to \PP^1$ as above are \emph{simply} ramified at distinct points away from $B$. 

It remains to promote the source and target of $f : D \to \PP^1$ to $\num{\ell}$- and $\num{n}$-marked curves. Endowing $D, \PP^1$ with stack structure to make $f$ \'etale, we must then trivialize the $\mu_2^k \times \mu_{m(\infty)}$ marked gerbes of $\PP^1$ as per our conventions in Remark \ref{rmk:trivialcostellogerbes}. On moduli, this is a gerbe of pure degree $\dfrac{1}{2^k m(\infty)}$. Ordering the $k$ images of the simple ramification points, each of their fibers, and the fibers of the marked points labelled by $J$ constitutes an $S_k \times (S_g)^k \times (S_g)^{\#J}$-torsor (the black points in Figure \ref{fig:excostellodiscretedata}). This gives the multiplicity $e$.

\end{proof}

Noam Elkies' response to \cite{380120} led to this lemma.

\begin{lemma}\label{lem:nospecialeffectivesubdivs}

Fix multiplicities $d : I \to \ZZ_{\geq 1}$ with $\sum d(i) = g+1$. General points $q_1, \dots, q_s \in D$ on a general smooth curve engender a divisor $B = \sum d(i)[q_i]$. There are no special effective subdivisors $0 \leq B' < B$ of degree $g$.

\end{lemma}

\begin{proof}

Fix numbers $b_1, \dots, b_s \in \NN$ adding up to $\sum b_i = g$ to obtain a map
\[D^s \to \Div^g; \quad \quad \quad (p_1, \dots, p_s) \mapsto \sum b_i [p_i].\]
The locus of special divisors is closed in $\Div^g$, as can be seen by applying upper semicontinuity theorem \cite[Theorem III.12.8]{hartshorne1977algebraic} to the universal sheaf $\OO(\mathscr{B})$ for the fibers of the projection $\pi : D \times \Div^g \to \Div^g$. 
We argue the pullback $P_{\{b_i\}} \subseteq D^s$ of the locus in $\Div^g$ of special divisors is a \emph{proper} closed subscheme. If $P_{\{b_i\}}$ contained the diagonal $\Delta_D$, a general point $p \in D$ would be a Weierstrass point. There are finitely many Weierstrass points on a curve over $\CC$, so $P_{\{b_i\}}$ is a proper closed subscheme. 

For any $1 \leq j \leq s$, we obtain a sequence
\[b_i := 
\left\{\begin{tikzcd}[row sep=small]
d(i)        &\text{if }i \neq j       \\
d(i)-1      &\text{if }i = j.
\end{tikzcd}.\right.\]
Our general points $q_1, \dots, q_s \in D$ are not in any $P_{\{b_i\}}$, so our divisor $B := \sum d(i) [q_i]$ contains no special effective divisors of degree $g$.

\end{proof}

\begin{remark}
Taking $C$ alone to be general in Costello's original proof \cite[Lemma 6.0.1]{Costello} does not suffice -- one must assume $I = \Supp D \subseteq C$ general as well to guarantee $\dim \Gamma(\OO(D')) = 1$ for \textit{any} divisor $0 \leq D' < D$.  Otherwise, take a generic genus two curve with $g^1_2$ mapping $f : C \to \PP^1$ and let $D = 2p + q$ with $p$ a Weierstrass point. If $D' = 2p$, $\dim \Gamma(\OO(D')) = 2$. 

\end{remark}

\begin{remark}
The pure degree $e$ is different from that computed by Costello:
\[e' = \dfrac{k!(g!)^{\# J} ((g-1)!)^k}{2^k m(\infty)}.\]
Consider the substack $\tilde{\frak M}_\Xi' \subseteq \tilde{\frak M}_\Xi(BS_d)$ ordering points of equal ramification separately by fixing $\tau$ as in Remark \ref{rmk:twistedstabmapsfixramification}. The restriction of $p'$ to $\tilde{\frak M}_\Xi'$ is of pure degree $e'$ because the simple ramification points must have the greatest label in their fiber for each of the $k$-marked points. Ordering the other unramified points is a $S_{g-1}$-torsor. The other terms count degree of the gerbes and ordering of the other points identically to Theorem \ref{thm:costellopushforwarddegreecalculation}.

\end{remark}

The main computation of virtual fundamental classes in \cite[Lemma 8.0.2]{Costello} thus applies to our above modifications: 
\[q_* \vfc{\MsSymG} = e \cdot \vfc{\overline{M}_{g, \ell-A}(X)}.\]

\begin{remark}
We employed the technology of \cite{twistedstablemaps} for convenience and brevity, but the same results may be achieved with Costello's original technology of weighted graphs. The data of a partial stabilization can be encoded on the level of graphs, and the stabilization of a map $C \to X$ can be reconstructed from the weighting of components by their curve classes in $X$. 
\end{remark}

\section{Applications of the pushforward formula}

This section addresses myriad articles which use Costello's Formula. The papers \cite{vakilgwthy}, \cite{pointedadmgcoversandgcfts}, \cite{genfnsforHHints}, \cite{obstnthiesandvfcs} reference but don't use Costello's Formula. The paper \cite{htpytypeofcobcatsofsurfs} uses other results from Costello's paper and not his formula, while \cite{gwthyofbddgerbesoverschs} uses it only for motivation.

The use of Costello's Formula in \cite{compthmsforgwisofsmpairs} will be addressed alongside other simplifications in forthcoming work by Sam Molcho, Rahul Pandharipande, and the authors. Similar techniques also apply to \cite{polyfamtautclasses} and \cite{stablemapstoratlcurves}, although both are subsumed by the suitably proper diagram in  \cite[\S 5.5]{logcptificationabeljacobi}.

\subsection{An algebraic proof of the hyperplane property of the genus-one GW-invariants of quintics}

The application of a ``cosection-localized version'' of Costello's Formula proposed in equation (1.4) is spelled out at the end of Section 2. There is a cartesian diagram
\[\begin{tikzcd}
D(\widetilde{\sigma}) \ar[r] \ar[d] \pb       &\widetilde{Y} \ar[r, "\widetilde{f}"] \ar[d, "p"] \pb      &\widetilde{X} \ar[r] \ar[d, "q"] \pb      &\widetilde{\mathcal{D}}  \ar[r] \ar[d] \pb     &\widetilde{\mathcal{M}}^w \ar[d]      \\
D(\sigma) \ar[r]       &Y \ar[r, "f"]      &X \ar[r]      &\mathcal{D}  \ar[r]       &\mathcal{M}^w,
\end{tikzcd}\]
where the map $\widetilde{\mathcal{M}}^w \to \mathcal{M}^w$ is a blowup and the obstruction theories of $Y, X$ relative to $\mathcal{D}$ pull back to those of $\widetilde{Y}, \widetilde{X}$ relative to $\widetilde{\mathcal{D}}$. Properness and birationality of the blowup lets one apply Costello's pushforward formula to show $q_\ast [\widetilde{X}]^{vir} = [X]^{vir}$.

One of two proofs they offer of Proposition 2.3 claims that $\widetilde{f}_*[\widetilde{Y}]_{loc}^{vir} = [\widetilde{X}]^{vir}$ and $f_*[Y]^{vir}_{loc} = [X]^{vir}$. From this claim and the valid application of Costello's pushforward formula, we see the Proposition is correct:
\[\deg [\widetilde{Y}]^{vir}_{loc} = \deg[Y]^{vir}_{loc}.\]

\subsection{Virtual pull-backs}

The final result \cite[Proposition 5.29]{Manolache} assumes the morphism is projective. This assumption was missing on the first version of the paper, however.

\subsection{Log Gromov-Witten theory with expansions}

The paper  only uses Costello's Formula to address pushforwards along logarithmic modifications that are pulled back from the target of the perfect obstruction theory \cite[Proposition 3.6.1]{loggwwithexps}. The definition of logarithmic modification includes a properness assumption \cite[\S 3.2]{loggwwithexps}.

\subsection{The cohomological crepant resolution conjecture for the Hilbert-Chow morphisms}

This paper uses Costello's Formula for a cartesian square
\[\begin{tikzcd}
\Msprels(\widehat{V}_1 \times_T \widehat{V}_2) \ar[r] \ar[d] \pb     &\Msprels(\widehat{V}_1) \times_T \Msprels(\widehat{V}_2) \ar[d]\\
T \times \mathcal{D}(d_1, d_2) \ar[r]         &T \times \mathcal{M}_{0, 3}(d_1) \times \mathcal{M}_{0, 3}(d_2)
\end{tikzcd}\]
in the proof of \cite[Lemma 5.5]{cohomcrepresconjhilbchow}. Immediately before, \cite[Lemma 5.4]{cohomcrepresconjhilbchow} shows that the lower horizontal arrow without $T$, $\mathcal{D}(d_1, d_2) \to \mathcal{M}_{0, 3}(d_1) \times \mathcal{M}_{0, 3}(d_2)$, is proper and birational.

\subsection{Gromov-Witten theory of \'etale gerbes, I: root gerbes}

Costello's result is used in \cite[Theorem 4.3]{gwthyofrootgerbes}. The morphism $Y^g_{0, n, \beta} \to \mathfrak M_{0, n, \beta}$ is an example of the Matsuki-Olsson construction, which is finite \cite[Theorem 4.1]{olssonmatsukikawamataviehweg}.

\subsection{The degeneration formula for logarithmic expanded degenerations}

The map $\mathfrak{T}^{u, spl}_0 \to \mathfrak{T}^u_0$ is observed to be a normalization in \cite[\S 7.2]{degfmlaforlogexpdegs}, subject to Remark \ref{rmk:normnisproper}. This proper map is the base of a diagram 
\[\begin{tikzcd}
K_\mathfrak{Q} \ar[r] \ar[d] \pb     &\mathfrak{Q}  \ar[r] \ar[d] \pb         &\mathfrak{Q}^{ext} \ar[r] \ar[d] \pb        &\mathfrak{T}^{u, spl}_0 \ar[d]        \\
K \ar[r]     &\mathfrak{T}^{etw}_{0}  \ar[r]       &\mathfrak{T}^{tw} \ar[r]        &\mathfrak{T}^u_{0}         
\end{tikzcd}\]
to which Chen applies Costello's Formula.

\subsection{Virtual classes of Artin stacks}

The result \cite[Theorem 5.2]{virtclassesartinstacks} includes a properness assumption.

\subsection{Virtual normalization and virtual fundamental classes}

Theorem 1 applies Costello's pushforward formula to a pullback along the map 
\[\widehat{Log} \subseteq Log^1 \to Log.\]
This pullback entails saturation of log structures, which is finite \cite[Proposition III.2.1.5 (2)]{ogusloggeom}.

\subsection{Orbifold techniques in degeneration formulas}

Costello's formula is used several times in \cite{orbifoldtechsindegfmlas}. 

\textit{Theorem 4.7}: the maps $\mathfrak{T}^{\mathfrak{r}'}_0 \to \mathfrak{T}^{\mathfrak{r}}_0$, $\mathscr{T}^{\mathfrak{r}'} \to \mathscr{T}^{\mathfrak{r}}$ along the bottom of the two squares written as one in Proposition 4.4.2 (2) are proper by Proposition 2.12.

\textit{Lemma 4.16}: the proof applies Costello's formula to the diagram
\[\begin{tikzcd}
K_\Xi \ar[r] \ar[d] \pb       &\prod K_{\Gamma_\nu}  \ar[d]    \\
    \mathscr{T}' \ar[r]        &(\mathscr{T})^h.
\end{tikzcd}\]
We need to argue $\mathscr{T}' \to (\mathscr{T}^{tw})^h$ is proper. We don't know how to define the contraction maps unless the rooting order is the same at each node, but this suffices.

Recall the description of $\Tc$ given in \cite[\S 5.2]{logcptificationabeljacobi}. The strict-\'etale topology on fs log schemes supports a sheaf of groups:
\[\Gt (S) := \Gamma(S, \overline{M}_S^{gp}).\]
An (oriented) \textit{tropical line} (\textit{bundle}) is a torsor in the strict-\'etale topology for $\Gt$. A map $S \to \Tc$ is a tropical line $P$ together with a subsheaf of sets $Q \subseteq P$ for which there locally exists a nonempty chain $\{\gamma_1 \leq \cdots \leq \gamma_n\} \subseteq \Gamma(S, \overline{M}_S^{gp})$ such that $Q$ is the subsheaf of sections of $P$ locally comparable to all the $\gamma_i$.

\begin{lemma}
The map $\mathscr{T}' \to (\mathscr{T}^{tw})^h$ is a log blowup, hence proper. 
\end{lemma}

\begin{proof}

For any map $S \to (\mathscr{T})^h$, take the fs pullback
\[\begin{tikzcd}
R \ar[r] \ar[d] \lpb       &S \ar[d]      \\
\mathscr{T}' \ar[r]        &(\mathscr{T})^h.
\end{tikzcd}\]
By strict-\'etale localization, assume each map $S \to \mathscr{T}$ corresponds to a $\mathbb{G}_m^{trop}$-torsor $P$ which is subdivided by sections $\gamma_1^j \leq \cdots \leq \gamma_{n_j}^j$. Use $\gamma_1^j$ to trivialize this torsor: 
\[P \simeq \mathbb{G}_m^{trop} \quad \quad \quad p \mapsto p-\gamma_1^j.\]
Write 
\[s_i^j = \gamma_i^j - \gamma_1^j \in \overline{M}_S\] 
for the image of $\gamma_i^j$ under this isomorphism.

In the present language, the map $\mathscr{T}' \to (\mathscr{T})^h$ corresponds to a subdivided tropical line $\gamma_1' \leq \cdots \leq \gamma_n'$ which induces all the others by forgetting some elements $\gamma_i'$. The first $\gamma_1'$ is never forgotten, since it corresponds to the unexpanded target $X$ in the expansion of $(X, D)$. Thus the element $\gamma_1'$ maps to 0 under our trivializations above. Write $s_i' := \gamma_i' - \gamma_1'$ similarly.

Take the fs product $B$ of all log blowups of $S$ at ideals given by pairs $(s_i^j, s_{i'}^{j'})$ for $1 \leq j \leq h$. Each of these blowups may be fs pulled back from the ideal of universal elements of $\overline{M}_{\mathscr{A}^2}$ on $\mathscr{A}^2$. A map $T \to S$ factors through $B$ (and uniquely) if and only if the set $(s_i^j|_T) \subseteq \overline{M}_T$ is totally ordered.

Since the $\gamma_i^j$'s all arise by forgetting parts of the subdivision $\gamma_1' \leq \cdots \leq \gamma_n'$, they are totally ordered on $R$. This means $R \to S$ factors through $B$. Observe also that the fs product $R \times_S B \to B$ is an isomorphism -- if the $s_i^j$'s are totally ordered, their sums with $\gamma_1$ yield a unique subdivision. Thus $R \to B$ is an isomorphism.

\end{proof}

\textit{Lemma 5.11}: the formula is applied to the cartesian diagram 
\[\begin{tikzcd}
K_\mathfrak{Q} \ar[r] \ar[d] \pb     &K  \ar[d]     \\
\mathfrak{Q} \ar[r] \ar[d] \pb        &\mathfrak{T}^{tw}_0  \ar[d]       \\
\mathfrak{T}^{u, spl}_0 \ar[r]         &\mathfrak{T}^u_0
\end{tikzcd}\]
beginning \S 5.4. They observe that the bottom horizontal arrow is a normalization of locally finite type stacks, hence subject to Remark \ref{rmk:normnisproper}. 

\textit{Lemma 5.12}: The bottom map in the diagram 
\[\begin{tikzcd}
K^{spl}_r \ar[r] \ar[d] \pb       &K_{\mathfrak{Q}_r} \ar[d]         \\
\mathfrak{T}^{r, spl}_0   \ar[r]       &\mathfrak{Q}_r
\end{tikzcd}\]
is the reduced induced closed substack, hence a proper map.

\textit{Lemma 5.15}: The map is a gerbe banded by $\mu_r$, which is proper.

\subsection{Birational invariance in log Gromov-Witten theory}

The paper \cite{biratlinvceinloggwthy} uses Costello's Pushforward Formula on the cartesian square $(1)$ in \cite[\S 1.6]{biratlinvceinloggwthy}:
\[\begin{tikzcd}
\Ms(Y) \ar[r] \ar[d] \pb      &\Ms(X) \ar[d]         \\
\Mprel'(\mathcal{Y}\to \mathcal{X}) \ar[r, "\Mprel(h)"]         &\Mprel(\mathcal{X}).
\end{tikzcd}\]

\begin{lemma}
The map $\Mprel(h)$ is proper. 
\end{lemma}

\begin{proof}

The map sends a square
\[\begin{tikzcd}
C  \ar[r] \ar[d]     &\mathcal{Y} \ar[d]        \\
\overline{C} \ar[r]        &\mathcal{X}
\end{tikzcd}\]
to the bottom horizontal arrow. Write $\overline{\mathscr{C}}$ for the universal curve on $\Mprel(\mathcal{X})$ and 
\[P := \overline{\mathscr{C}} \times_\mathcal{X} \mathcal{Y}\]
for the pullback. Then $\Mprel(h)$ factors through the inclusion of components of $\Ms(P/\Mprel(\mathcal{X}))$ on which the universal map $C \to P \to \overline{\mathscr{C}}$ is a partial stabilization and Lemma \ref{lem:partstabnclopen} concludes. 

\end{proof}

\subsection{Relative and Orbifold Gromov-Witten Invariants}

In \cite[Diagram 2.3.1]{relandorbgwinvts}, we see another application of Costello's pushforward formula. This square is a special case of a more general class of diagrams investigated in Section 7.3:
\[\begin{tikzcd}
\Ms_\Gamma^{tr}(\mathscr{X}^{rel}/\mathscr{T}) \ar[r, "\phi_{\mathscr{X}}"] \ar[d] \pb      &\Ms_\Gamma(\mathscr{X}) \ar[d]       \\
\Mprel^{rel}_{0,n}(\mathscr{A}, B\mathbb{G}_m)' \ar[r, "\phi_{\mathscr{A}}"]         &\Mprel_{0, n}(\mathscr{A})'.
\end{tikzcd}\]

The stack $\Mprel_{0, n}(\mathscr{A})'$ is an open substack of $\Mprel_{0,n}(\mathscr{A})$.

\begin{lemma}\label{lem:relandorbgwischeckproperviacostellomethod}
The map 
\[\Mprel^{rel}_{0,n}(\mathscr{A}, B\mathbb{G}_m)' \to \Mprel_{0, n}(\mathscr{A})'\]
is proper. 
\end{lemma}

\begin{proof}

This map is pulled back from the map 
\[\Mprel^{rel}_{0,n}(\mathscr{A}, B\mathbb{G}_m)^* \to \Mprel_{0, n}(\mathscr{A}).\]
This map sends a square 
\[\begin{tikzcd}
C \ar[r] \ar[d]      &\widetilde{\mathscr{A}}_r \ar[d]        \\
\overline{C}  \ar[r]      &\mathscr{A}_r \times S
\end{tikzcd}\]
to the lower horizontal arrow. We again employ Lemma \ref{lem:partstabnclopen} by describing this map as the locus among relative moduli of stable curves where a particular morphism is a partial stabilization. 

\end{proof}

The same techniques handle the square \cite[7.1.2]{relandorbgwinvts}:
\[\begin{tikzcd}
\Ms^{rel}_{g=0}(X_r, D_r) \ar[r] \ar[d] \pb         &\Ms_{g=0}^{orb}(X_r) \ar[d]       \\
\Mprel_{g=0}^{rel}(\mathscr{A}_r, \mathscr{D}_r) \ar[r]          &\Mprel_{g=0}^{orb}(\mathscr{A}_r). 
\end{tikzcd}\]

We still must address the map $\phi_\mathscr{A}$ in
\[\begin{tikzcd}
\Ms^{rel}(X_r, D_r) \ar[r, "\phi_{\mathscr{X}}"] \ar[d] \pb      &\Ms^{rel}(X,D) \ar[d]       \\
\Mprel^{rel}_{0,n}(\mathscr{A}_r, \mathscr{D}_r) \ar[r, "\phi_{\mathscr{A}}"]         &\Mprel^{rel}_{0, n}(\mathscr{A}, \mathscr{D}).
\end{tikzcd}\]

\begin{remark}
No stabilization occurs in $\phi_\mathscr{X}$. 

\end{remark}

\begin{lemma}
The map 
\[\phi_{\mathscr{A}} : \Mprel^{rel}_{0, n}(\mathscr{A}_r, \mathscr{D}_r) \to \Mprel^{rel}_{0, n}(\mathscr{A}, \mathscr{D})\]
is of pure degree 1. 
\end{lemma}

\begin{proof}

Write $u : \mathfrak D \times_{\widetilde{\mathscr{A}}} \widetilde{\mathscr{A}_r} \to \Mprel(\mathscr{A}, \mathscr{D})$ for the pullback of the universal curve along the map between universal expansions. The space $\Mprel^{rel}(\mathscr{A}_r, \mathscr{D}_r)$ lies inside the spaces of relative stable map $\Ms(u)$ as the locus with $S$-points where $C \to \mathfrak{D}|_S$ is an isomorphism; denote its closure by $\overline{\Mprel}$. Then $\overline{\Mprel} \to \Mprel(\mathscr{A}, \mathscr{D})$ is proper and birational, so restriction to the dense open $\Mprel(\mathscr{A}_r, \mathscr{D}_r)$ is pure degree~1. 

\end{proof}

\appendix

\section{Degree of a Generically Finite Morphism}\label{sec:degofgenfin}

The stacks project offers two definitions of generic finiteness. We assume our stacks are locally noetherian and elaborate on definition (1) of \cite[073A]{stacks-project}. 

\begin{definition}\label{def:repablegenfindeg}
Let $f : X \to Y$ be locally of finite type and $\eta \in Y$ be a maximal point. We say $f$ is \textit{generically finite at $\eta$} if the preimage $X \times_Y \eta$ is a finite, \textit{nonempty} set. Equivalently, there's an affine open $V \subseteq Y$ and finitely many $U_1, \dots, U_n$ such that $U_i \to V$ is finite and $\eta \in V$ and $X \times_Y \eta \subseteq \bigcup_n U_i$ \cite[02NW]{stacks-project}. 

Given that $f : X \to Y$ is generically finite at some maximal $\eta$, we say it is \textit{of degree $d$} at $\eta$ if \cite[02NY]{stacks-project}
\[d = \sum_{\xi \in f^{-1}(\eta)} \dim_{R(\eta)} \OO_{X, \xi}.\]

A morphism $f : X \to Y$ locally of finite type is said to be \textit{generically finite} or \textit{of degree $d$} if it is so at every maximal point $\eta \in Y$. 

A representable morphism $X \to Y$ locally of finite type between algebraic stacks is said to be generically finite or of degree $d$ (at a specific maximal point $\eta \in Y$ or for all) if the same is true for pulling back along some smooth cover $V \to Y$ by a scheme (with $\xi \in V$ mapping to $\eta$).

\end{definition}

\begin{remark}
Generically finite and degree $d$ both pull back along flat, quasicompact morphisms $Y' \to Y$ and may be checked after some (equivalently any) flat, quasicompact cover. This is because generalizations lift along flat, quasicompact morphisms, ensuring that maximal points map to each maximal point. 
\end{remark}

\begin{lemma}
Let $X \to \Spec k$ be a finite morphism from a DM stack to a field. Then $X$ admits a finite \'etale cover from a scheme.
\end{lemma}

\begin{proof}

Pick a finite type \'etale cover $P \to X$. Then $P \to X$ is locally quasifinite \cite[03WS]{stacks-project} and hence quasifinite \cite[01TD]{stacks-project}. The composite $P \to \Spec k$ is quasifinite, hence finite \cite[02NH]{stacks-project}. The map $P \to X$ is then finite. 

\end{proof}

\begin{definition}
A finite DM-type morphism $X \to \Spec k$ is of pure degree $d$ if, for some (equiv. any) finite \'etale cover $P \to X$ by a scheme, 
\[\dfrac{\deg (P/\Spec k)}{\deg (P/X)} = d.\]

A DM-type morphism $X \to Y$ of locally noetherian artin stacks is \textit{generically finite} if, for all maximal points $\eta \to Y$, the pullback 
\[X \times_Y \eta \to \eta\]
is finite. 
\end{definition}

\begin{remark}
The definition of \textit{degree $d$} for generically finite morphisms is determined by its properties: 
\begin{itemize}
    \item A composite $X \overset{f}{\to} Y \overset{g}{\to} Z$ for which $\deg f$, $\deg g$, $\deg g \circ f$ are well defined satisfies
    \[\deg (g \circ f) = \deg f \cdot \deg g.\]
    \item Given a pullback square 
    \[\begin{tikzcd}
    X' \ar[r] \ar[d, "f'"] \pb      &X \ar[d, "f"]      \\
    Y' \ar[r]      &Y
    \end{tikzcd}\]
    with $Y' \to Y$ flat and quasicompact, $f$ is generically finite (of degree $d$) if and only if $f'$ is. 
    \item Agreement with the notion for representable morphisms in Definition \ref{def:repablegenfindeg}. 
\end{itemize}
\end{remark}

We conclude with two folklore observations that we use in the body of the text.

\begin{remark}[``Stability is an open condition'']\label{rmk:stabilityisopen}
Suppose $f : X \to Y$ is locally finite type and $X, Y$ are algebraic stacks. There is a substack $U \subseteq X$ representing morphisms $T \to X$ such that $f|_T$ is DM type, and this substack is open. A map is DM type when the diagonal is unramified, which is an open condition by \cite[0475]{stacks-project}.

\end{remark}

This shows that the locus where a family of prestable maps is stable is open in the base.

\begin{remark}\label{rmk:normnisproper}
If $X$ is an algebraic stack locally of finite type, then its normalization $X^\nu \to X$ is finite. This is because normalizations are integral \cite[035Q]{stacks-project} and the map is locally of finite type \cite[01WJ]{stacks-project}. 
\end{remark}

\bibliographystyle{alpha}
\bibliography{costello}

\end{document}